\documentclass[11pt,reqno]{amsart}
\usepackage{amsmath}
\usepackage{amsfonts}

\numberwithin{equation}{section}
\usepackage{times}
\usepackage{amsmath,amsfonts,amstext,amssymb,amsbsy,amsopn,amsthm,eucal}
\usepackage{txfonts}
\usepackage{dsfont}
\usepackage{graphicx}   
\usepackage{hyperref}
\usepackage{accents}
\usepackage{enumerate}
\usepackage{xcolor}

\usepackage{verbatim}

\newcommand{\Ric}{{\rm Ric}}

\newcommand{\Vol}{{\rm Vol}}

\newcommand{\cH}{\mathcal{H}}

\newcommand{\cM}{\mathcal{M}}

\newcommand{\cU}{\mathcal{U}}

\newtheorem{theorem}{Theorem}[section]

\newtheorem{proposition}[theorem]{Proposition}
\newtheorem{lemma}[theorem]{Lemma}
\newtheorem{corollary}[theorem]{Corollary}
\theoremstyle{definition}
\newtheorem{definition}[theorem]{Definition}
\theoremstyle{remark}
\newtheorem{remark}{Remark}[section]
\theoremstyle{remark}

\theoremstyle{remark}

\theoremstyle{remark}\newtheorem{conjecture}{Conjecture}[section]
\theoremstyle{remark}

\def\comp{\ensuremath\mathop{\scalebox{.6}{$\circ$}}}

\begin{document}

\title{Weak scalar curvature lower bounds along Ricci flow}
\date{\today}

\author{Wenshuai Jiang}
\address{W. Jiang,~~School of mathematical Sciences,  Zhejiang University}
\email{wsjiang@zju.edu.cn}
\author{Weimin Sheng}
\address{W. Sheng,~~School of mathematical Sciences,  Zhejiang University}
\email{weimins@zju.edu.cn}
\author{Huaiyu Zhang}
\address{H. Zhang,~~School of mathematical Sciences,  Zhejiang University}
\email{zhanghuaiyu\_math@outlook.com}
\thanks{The first author was supported by NSFC (No. 12071425, No. 12125105); The second author was supported by NSFC (No. 11971424, No. 12031017); The third author was supported by NSFC (No. 11971424)}

\begin{abstract}


In this paper, we study Ricci flow on compact manifolds with a continuous initial metric. It was known from Simon that the Ricci flow exists for a short time. We prove that the scalar curvature lower bound is preserved along the Ricci flow if the initial metric has a scalar curvature lower bound in distributional sense provided that the initial metric is $W^{1,p}$ for some $n<p\le \infty$. 
	As an application, we use this result to study the relation between Yamabe invariant and Ricci flat metrics.
We prove that if the Yamabe invariant is nonpositive and the scalar curvature is nonnegative in distributional sense, then the manifold is isometric to a Ricci flat manifold.

\end{abstract}
\maketitle
\tableofcontents

\section{Introduction}
The study of low-regularity Riemannian metrics with some weak curvature conditions is an important theme in Riemannian geometry. For sectional curvature lower bounds and Ricci curvature lower bounds, many beautiful results have been established (for Alexandrov spaces theory, see, e.g. \cite{ABN}, \cite{BBI}; for Ricci curvature lower bounds, see, e.g. \cite{CC1}, \cite{CC2}, \cite{CC3}, \cite{CN1}, \cite{CN2},  \cite{JN21}, \cite{CJN21}; or for an optimal transport approach, see, e.g. \cite{LV}, \cite{St1}, \cite{St2}, \cite{St3}). 

However, it has not been well understood for scalar curvature lower bounds. One of a fundamental property for classical scalar curvature lower bounds is that they are  preserved by Ricci flow. Ricci flow is firstly introduced by Hamilton \cite{Hamilton1982}, which comes to be an powerful tool in geometry and has been used to prove the Poincar\'e Conjecture, see \cite{Pr1}, \cite{Pr2}, \cite{Pr3}, \cite{BL}, \cite{CZ}, \cite{MT}. The lower bounds preserving property is still true for scalar curvature lower bounds in some weak sense and it is quite useful in the study of low-regularity Riemannian metrics with some weak scalar curvature conditions. Gromov introduced a scalar curvature lower bound in some weak sense, which could be defined for $C^0$ metrics, see \cite{Gm}. Bamler and Burkhardt-Guim considered it by a Ricci flow approach, they elaborated on defining a weak sense scalar curvature lower bound which is required to be  preserved by Ricci flow, see\cite{Bm}, \cite{PBG}. If a metric satisfies that its Ricci flow has a scalar curvature lower bound at each short positive time, then this metric must satisfies the definition of scalar curvature lower bounds in Gromov \cite{Gm}. It remains an open problem whether the converse is true (see \cite{PBG}). Schoen also considered scalar curvature on low-regularity metrics. He proposed a question that if the Yamabe invariant $\sigma(M)$ is nonpositive, the metric admits singularity in a subset and the scalar curvature is at least $\sigma(M)$ away from the singular set, then whether we can prove that the metric is smooth and Ricci flat provided that the singular set is small, see \cite{LiMa}. Li and Mantoulidis \cite{LiMa} gave an answer for his skeleton metrics and for $3$-manifolds with metrics admitting point singularities. For more related results, see the survey of Sormani \cite{So}.

McFeron and Sz\'ekelyhidi \cite{McSz} proved that if a Lipschitz metric is smooth away from a hypersurface satisfying certain conditions on the mean curvature and the metric has nonnegative scalar curvature pointwisely away from the hypersurface, then the scalar curvature will be nonnegative pointwisely  on the whole manifold under Ricci flow. This property appears to be quite useful in the study of the positive mass theorem, especially for the rigidity part of the positive mass theorem, see \cite{McSz}. (See also \cite{Miao2002}.) Shi and Tam \cite{ShTa} have proved the nonnegative scalar curvature preserving property for $W^{1,p}$ $(n<p\le \infty)$ metrics which are smooth away from a singular set with Minkowski dimension at most $n-2$. Moreover, they apply this property to get a positive mass theorem for low-regularity metrics and establish an answer of Schoen's question for $W^{1,p}$ metrics. LeFloch and Mardare defind a scalar curvature lower bound in distributional sense for $W^{1,p}$ metrics. This was further studied by \cite{Lee2015a,LS15} in asymptotically flat manifolds. Lee and Lefloch \cite{Lee2015a} proved a positive mass theorem on spin manifolds with $W^{1,n}$ metrics which has nonnegative scalar curvature in distributional sense. Lee and Lefloch's (\cite{Lee2015a}) definition of weak scalar curvature lower bounds recovers the case in \cite{McSz}, \cite{Miao2002} and \cite{ShTa}, thus the positive mass theorem in \cite{Lee2015a} recovers it in \cite{Miao2002}, \cite{McSz} and \cite{ShTa} in spin case.

We improve some of the above results. The main theorem is the following:
\begin{theorem}\label{thm1.1}
Let $M^n$ be a compact manifold with a metric $g\in W^{1,p}(M)$ $(n< p\le \infty)$. Assume the scalar curvature $R_g\ge a$ for some constant $a$ in distributional sense, and let $g(t),t\in(0,T_0]$ be the Ricci flow starting from the metric $g$. Then for any $t\in(0,T_0]$, there holds $R_{g(t)}\ge a$ on $M$. 
\end{theorem}
{
In \cite{PBG}, Burkhardt-Guim introduced a weak scalar curvature lower bound for $C^0$-metric using Ricci flow. By comparing with the definition in \cite{PBG} and Theorem \ref{thm1.1},  we get:
\begin{corollary}
Let $M^n$ be a compact manifold with a metric $g\in W^{1,p}(M)$ $(n< p\le \infty)$. Assume the scalar curvature $R_g\ge a$ for some constant $a$ in distributional sense, then the scalar curvature $R_g\ge a$ in the sense of \cite{PBG}.
\end{corollary}}
As a consequence of Theorem \ref{thm1.1}, we can prove the following  result:

\begin{theorem}\label{thmY2}
Let $(M^n,g)$ ($n\ge 3$) be a compact manifold with $g\in W^{1,p}(M)$ $(n< p\le \infty)$. Assume the distributional scalar curvature $R_g\ge 0$, then either $(M,g)$ is isometric to a Ricci flat manifold or there exists a smooth positive scalar curvature metric on $M$.
\end{theorem}

\begin{remark}
By Theorem \ref{thmY2}, if $M$ is a topological $n$-torus with a metric $g\in W^{1,p}(M)$ for some $n<p\le \infty$, then $(M,g)$ is isometric to a flat torus provided that $g$ has nonnegative distributional scalar curvature. More generally, if $M$ is a differential manifold with the Yamabe invariant $\sigma(M)\le 0$ and with a metric $g$ satisfying the same condition as above, then $(M,g)$ is isometric to a Ricci flat manifold.
\end{remark}

\begin{remark}
Bourguignon, Gromov, Lawson, Schoen and Yau  have established relevant results in smooth version, see \cite{GL80}, \cite{ScYa2}, \cite{KW75}.
\end{remark}

\begin{remark}\label{rmk1.1}
We can check that if $g$ is $C^2$ away from a closed subset $\Sigma$ with its Hausdorff measure $\cH^{n-\frac{p}{p-1}}(\Sigma)<\infty$ when $n<p<\infty$ or $\cH^{n-1}(\Sigma)=0$ when $p=\infty$, and if $R_g\ge a$ pointwisely on $M\setminus\Sigma$, then $R_g\ge  a$ in distributional sense. Though it is proved in \cite{JSZ}, for the sake of completeness, we give a proof in appendix. Also, if $g$ is $C^2$ away from a hypersurface whose mean curvature taken with respect to the metric in the interior is at least it taken with respect to the metric in the exterior and $R_g\ge a$ away from the hypersurface, then  $R_g\ge  a$ in distributional sense. It has been proved in \cite[Proposition 5.1]{Lee2015a}.
\end{remark}

Therefore, as an application, we can get similar results for metrics admitting singularity in a subset, which are extensions of some results in \cite{ShTa}. Specifically, if $g$ is $C^2$ away from a closed subset $\Sigma$ with $\cH^{n-\frac{p}{p-1}}(\Sigma)<\infty$ when $n<p<\infty$ or $\cH^{n-1}(\Sigma)=0$ when $p=\infty$, and if $R_g\ge a$ or $R_g\ge 0$ pointwisely on $M\setminus\Sigma$, then the above theorems still hold.

\vskip 3mm
\noindent
\textbf{Organization:}

In section 2, we will mollify the metric by convolution, and we also provide an estimate of the weak scalar curvature in this smooth approximation (see Lemma \ref{lm2.2}).
In section 3, we recall the definition of the Ricci flow and the Ricci-DeTurck flow, and we obtain some estimates for them. The main estimate we need is Theorem \ref{mainsec6}.
In section 4, we reacall the definition of the conjugate heat equation and prove some property of the solution of this equation (see Proposition \ref{prop_varphit}), which will be needed in section 5.
In section 5, we prove Theorem \ref{thm1.1} and Theorem \ref{thmY2}.
In section 6, we state some further questions.

\vskip 3mm
\noindent
\textbf{Acknowlegments:} The authors would like to thank Prof. Dan Lee and Prof. C. Sormani for many helpful suggestions.

\section{Approximation of Singular Metrics}
Let $M^n$ be a compact smooth manifold with a metric $g\in C^0(M)\cap W^{1,2}(M)$  
 and let $h$ be a smooth metric on $M$ such that $C^{-1}h\le g\le Ch$ for some constant $C>1$.  Throughout this paper, $\tilde\nabla$ and the dot product will denote the Levi-Civita connection and the inner product taken with respect to $h$ respectively,  We define the distirbutional scalar curvature  of $g$ as in \cite{LM07,Lee2015a,LS15},
\begin{align}
	\langle R_g,\varphi\rangle :=\int_M \left(-V\cdot \tilde\nabla \left(\varphi \frac{d\mu_g}{d\mu_h}\right)+F\varphi \frac{d\mu_g}{d\mu_h}\right)d\mu_h, \forall \varphi\in C^\infty(M),
\end{align}
where $\mu_h$ is the Lebesgue measure taken with respect to $h$, $V$ is a vector field and $F$ is a scalar field, defined as:
\begin{align}
\Gamma_{ij}^k&:=\frac{1}{2}g^{kl}\left(\tilde\nabla_i g_{jl}+\tilde\nabla_j g_{il}-\tilde\nabla_l g_{ij}\right), \\
	V^k&:=g^{ij}\Gamma_{ij}^k-g^{ik}\Gamma_{ji}^j=g^{ij}g^{k\ell}(\tilde\nabla_jg_{i\ell}-\tilde\nabla_{\ell}g_{ij}),\\
	F&:= {\rm{tr}}_g\tilde{\Ric}-\tilde \nabla_kg^{ij}\Gamma_{ij}^k+\tilde\nabla_kg^{ik}\Gamma_{ji}^i+g^{ij}\left(\Gamma_{k\ell}^k\Gamma_{ij}^\ell-\Gamma_{j\ell}^k\Gamma_{ik}^\ell\right),
\end{align}
where $\tilde{\Ric}$ is the Ricci tensor of $h$. By \cite{Lee2015a}, $\langle R_g,\varphi\rangle$ coincides with the integral $\int_M R_g\varphi d\mu_g$ in classical sense if $g\in C^2(M)$ and $\langle R_g,\varphi\rangle$ is independent of $h$ for any $g\in C^0(M)\cap W^{1,2}(M)$. For more details about the distirbutional scalar curvature and related results, see \cite{Lee2015a},{\cite{LM07},\cite{LS15}}, \cite{JSZ}. 

Let $a$ be some constant, we say that the distirbutional scalar curvature  of $g$ is at least $a$ if  $\langle R_g,\varphi\rangle-a\int_M \varphi d\mu_g\ge 0$ for any nonnegative function $\varphi\in C^\infty(M)$. We will abbreviate this inequality as $\langle R_g-a,\varphi\rangle\ge 0$.

The following mollification lemma has been proved in \cite[Lemma 4.1]{grant2014positive}, by a standard convolution mollifying procedure and a use of partition of unity. Although their lemma is a $W^{2,\frac{n}{2}}$ version, our version could be proved in the same way.
\begin{lemma}[\cite{grant2014positive}]\label{lm2.1}
Let $M^n$ be a compact smooth manifold and $g$ be a $C^0\cap W^{1,p}$ ($1\le p\le\infty$) metric on $M$, then there exists a family of smooth metric $g_\delta$, $\delta>0$, such that $g_\delta$ converge to $g$ both in $C^0$-norm and in $W^{1,p}$-norm as $\delta\to 0^+$.
\end{lemma}

Under this mollification, we have an estimate of the distirbutional scalar curvature. The following lemma has been essentially proved in \cite{JSZ}, though \cite{JSZ} only gives a proof for the estimate of $\langle R_{g_ \delta},u^2\rangle$, the following lemma can be proved in the same way. In order to be self-contained, we give a proof here.
\begin{lemma}[\cite{JSZ}]\label{lm2.2}
	Let $M^n$ be a compact smooth manifold and $g$ be a $C^0\cap W^{1,n} $ metric on $M$. Let $g_\delta$ be the mollification in Lemma \ref{lm2.1}, then we have that for any $\epsilon>0$, there exists $\delta_0=\delta_0(g)>0$, such that 
\[|\langle R_{g_ \delta},u\rangle-\langle R_g,u\rangle|\leq \epsilon\|u\|_{W^{1,\frac{n}{n-1}}(M)},\forall u\in C^\infty(M),\forall \delta\in (0,\delta_0).\]
where $R_{g_\delta}$ is the scalar curvature of $g_\delta$.
\end{lemma}
 
\begin{proof}
Let $h$ be a smooth metric on $M$ with $C^{-1}h<g<Ch$, here and below $C$ will denote some positive constant which depends on $n$ and the Sobolev constant of $g$, but independent of $\delta$ and varis from line to line. Let $V_\delta$ and $F_\delta$ be the vector field and scalar field  in the definition of distirbutional scalar curvature of $g_\delta$. Then we have 
\begin{align}
	\lim_{\delta\to0^{+}}\left(\int_M|V_\delta-V|^nd\mu_h+\int_M|F_\delta-F|^{n/2}d\mu_h\right)=0.
\end{align}	
We also have
\begin{align}
	\lim_{\delta\to0^{+}} \int_M\left|\tilde\nabla \frac{d\mu_{g_\delta}}{d\mu_h}-\tilde \nabla \frac{d\mu_{g}}{d\mu_h}\right|^n d\mu_h =0.
\end{align}	

Using triangular inequality and H\"older inequality, we can calculate that 
\begin{align*}
	&\Big|\int_M F_\delta u  \frac{d\mu_{g_\delta}}{d\mu_h}d\mu_h-\int_M Fu \frac{d\mu_{g}}{d\mu_h}d\mu_h\Big|\\
	&\le \int_M \Big| F_\delta u \frac{d\mu_{g_\delta}}{d\mu_h}-Fu \frac{d\mu_{g}}{d\mu_h}\Big|d\mu_h\\
	&\le \int_M|F_\delta u -Fu |\Big|\frac{d\mu_{g_\delta}}{d\mu_h}\Big| d\mu_h+\int_M |Fu| \Big| \frac{d\mu_{g_\delta}}{d\mu_h}-\frac{d\mu_{g}}{d\mu_h}\Big|d\mu_h\\
	&\le C \int_M|F_\delta u -Fu | d\mu_h+\sup_M\Big| \frac{d\mu_{g_\delta}}{d\mu_h}-\frac{d\mu_{g}}{d\mu_h}\Big| \int_M |Fu| d\mu_h\\
&\le C\left(\int_M |F_\delta-F|^{n/2}d\mu_h \right)^{2/n}\left(\int_M|u|^{ n/(n-2)}d\mu_h\right)^{(n-2)/n}+\sup_M\Big| \frac{d\mu_{g_\delta}}{d\mu_h}-\frac{d\mu_{g}}{d\mu_h}\Big| \left(\int_M |F|^{n/2}d\mu_h\right)^{2/n}\left(\int_M |u|^{ n/(n-2)}d\mu_h\right)^{(n-2)/n}.
\end{align*}	
By Sobolev inequality
\[\left(\int_M |u|^{ n/(n-2)}d\mu_h\right)^{(n-2)/n}\le C\|u\|_{W^{1,\frac{n}{n-1}}(M)},\]
we get that
\begin{align*}
&\Big|\int_M F_\delta u  \frac{d\mu_{g_\delta}}{d\mu_h}d\mu_h-\int_M Fu \frac{d\mu_{g}}{d\mu_h}d\mu_h\Big|\\
&\le \left(C\left(\int_M |F_\delta-F|^{n/2}d\mu_h \right)^{2/n}+C\sup_M\Big| \frac{d\mu_{g_\delta}}{d\mu_h}-\frac{d\mu_{g}}{d\mu_h}\Big| \left(\int_M |F|^{n/2}d\mu_h\right)^{2/n}\right) \|u\|_{W^{1,\frac{n}{n-1}}(M)}.
\end{align*}

	Similarly, for the term involving $V$, we can calculate that
	\begin{align}
		\Big|\int_M &V\cdot \tilde\nabla \left(u\frac{d\mu_g}{d\mu_h}\right)d\mu_h-\int_M V_\delta\cdot \tilde\nabla \left(u\frac{d\mu_{g_\delta}}{d\mu_h}\right)d\mu_h\Big|\\
		\le & \int_M |V-V_\delta|\cdot\Big|\tilde\nabla \left(u\frac{d\mu_{g_\delta}}{d\mu_h}\right)\Big|d\mu_h+\int_M |V|\cdot\Big|\tilde\nabla \left(u\frac{d\mu_{g}}{d\mu_h}-u\frac{d\mu_{g_\delta}}{d\mu_h}\right)\Big|d\mu_h\\
		\le & \left(\int_M|V-V_\delta|^nd\mu_h\right)^{1/n}\left(\int_M\Big|\tilde\nabla \left(u\frac{d\mu_{g_\delta}}{d\mu_h}\right)\Big|^{n/(n-1)}d\mu_h\right)^{(n-1)/n}\\
		&+ \left(\int_M|V|^nd\mu_h\right)^{1/n}\left(\int_M\Big|\tilde\nabla \left(u\frac{d\mu_{g_\delta}}{d\mu_h}-u\frac{d\mu_{g}}{d\mu_h}\right)\Big|^{n/(n-1)}d\mu_h \right)^{(n-1)/n}.
	\end{align}
	Using Sobolev inequality and H\"older inequality again, we have
	\begin{align}
		\int_M&\Big|\tilde\nabla \left(u\frac{d\mu_{g_\delta}}{d\mu_h}\right)\Big|^{n/(n-1)}d\mu_h\\
		\le &C(n)\int_M\left|\tilde\nabla u \frac{d\mu_{g_\delta}}{d\mu_h}\right|^{n/(n-1)}d\mu_h+C(n)\int_M\left| u\tilde\nabla \frac{d\mu_{g_\delta}}{d\mu_h}\right|^{n/(n-1)}d\mu_h\\
		\le & C\int_M |\tilde\nabla u|^{n/(n-1)} d\mu_h+C\left(\int_M |u|^{ n/(n-2)}d\mu_h\right)^{(n-2)/(n-1)}\left(\int_M|\tilde\nabla \frac{d\mu_{g_\delta}}{d\mu_h}|^nd\mu_h\right)^{1/(n-1)}\\
		\le &C\int_M |\tilde\nabla u|^{n/(n-1)} d\mu_h+C\|u\|_{W^{1,\frac{n}{n-1}}(M)}^{n/(n-1)}\left(\int_M|\tilde\nabla \frac{d\mu_{g_\delta}}{d\mu_h}|^nd\mu_h\right)^{1/(n-1)}\\
		\le &C\left(1+\int_M|\tilde\nabla \frac{d\mu_{g_\delta}}{d\mu_h}|^nd\mu_h\right)^{1/(n-1)} \|u\|_{W^{1,\frac{n}{n-1}}(M)}^{n/(n-1)},
	\end{align}
	and we also have, 
	\begin{align}
	\int_M&\Big|\tilde\nabla \left(u\frac{d\mu_{g_\delta}}{d\mu_h}-u\frac{d\mu_{g}}{d\mu_h}\right)\Big|^{n/(n-1)}d\mu_h\\
	\le & C\sup_M 	\Big|\frac{d\mu_{g_\delta}}{d\mu_h}- \frac{d\mu_{g}}{d\mu_h}\Big|^{n/(n-1)}\int_M |\tilde\nabla u|^{n/(n-1)} d\mu_h\\
	&+ C \left(\int_M |u|^{ n/(n-2)}d\mu_h\right)^{(n-2)/(n-1)}\left(\int_M|\tilde\nabla \left(\frac{d\mu_{g_\delta}}{d\mu_h}-\frac{d\mu_{g}}{d\mu_h}\right)|^nd\mu_h\right)^{1/(n-1)}\\
	\le &C\left(\sup_M 	\Big|\frac{d\mu_{g_\delta}}{d\mu_h}- \frac{d\mu_{g}}{d\mu_h}\Big|^{n/(n-1)}
	+ \left(\int_M|\tilde\nabla \left(\frac{d\mu_{g_\delta}}{d\mu_h}-\frac{d\mu_{g}}{d\mu_h}\right)|^nd\mu_h\right)^{1/(n-1)}\right)\|u\|_{W^{1,\frac{n}{n-1}}(M)}^{n/(n-1)}.
	\end{align}
	We combine these estimates, and then we get
	\begin{align}
		\Big|\int_M& V\cdot \tilde\nabla \left(u\frac{d\mu_g}{d\mu_h}\right)d\mu_h-\int_M V_\delta\cdot \tilde\nabla \left(u\frac{d\mu_{g_\delta}}{d\mu_h}\right)d\mu_h\Big|\\
		\le & C\left(\int_M|V_\delta-V|^nd\mu_h\right)^{1/n}\left( 1+\left(\int_M|\tilde\nabla \frac{d\mu_{g_\delta}}{d\mu_h}|^nd\mu_h\right)^{1/n}\right)\|u\|_{W^{1,\frac{n}{n-1}}(M)}\\
		&+ C\left( \sup_M 	\Big|\frac{d\mu_{g_\delta}}{d\mu_h}-\frac{d\mu_{g}}{d\mu_h}\Big| 
	+  \left(\int_M|\tilde\nabla \left(\frac{d\mu_{g_\delta}}{d\mu_h}-\frac{d\mu_{g}}{d\mu_h}\right)|^nd\mu_h\right)^{1/n}\right)\|u\|_{W^{1,\frac{n}{n-1}}(M)}.
	\end{align}
	Therefore, for any $\epsilon>0$, there exists $\delta_0>0$ small enough, such that
	\begin{align}
		|\langle R_{g_\delta},u\rangle-\langle R_g,u\rangle|\le \epsilon \|u\|_{W^{1,\frac{n}{n-1}}(M)}, \forall u\in C^\infty(M), \forall \delta\in (0,\delta_0),
	\end{align}
	whcih completes the proof of the lemma.
	
\end{proof}

\section{Estimates on Ricci flow}\label{sec2.3}
The Ricci flow was introduced by Hamilton\cite{Hamilton1982}. Its definition is as follows:
\begin{definition}[Ricci flow]The Ricci flow on $M$ is a family of time dependent metrics $g(t)$ such that \[\frac{\partial}{\partial t}g(t)=-2\Ric_{g(t)},\] where ${\Ric_{g}}$ is the Ricci curvature tensor of $g$.  
\end{definition}
The main theorem in this section is the following:
\begin{theorem}\label{thm66.2} \label{mainsec6}
There exists an $\epsilon(n)>0$ such that, for any compact $n$-manifold $M$ with a $ W^{1,p}$ metric $\hat g$ with $n<p\le +\infty$, there exists a $T_0=T_0(n,g)>0$ and a family of metrics $g(t)\in C^{\infty}(M\times(0,T_0]),t\in (0,T_0]$ which solves Ricci flow for $t\in(0,T_0]$, and satisfies
\begin{itemize}
\item[(1)]$\lim_{t\to 0}d_{GH}((M,g(t)),(M,\hat g))=0.$
\item[(2)] $|{\rm{Rm}}(g(t))|(t)\leq \frac{C(n,\hat g,p)}{t^{\frac{n}{4p}+\frac{3}{4}}}$, $\forall t\in(0,T_0]$.
\item[(3)] $\int_0^{T_0}\int_M |{\rm{Rm}}( g(t))|^2d\mu_{g(t)}dt\le C(n,\hat g,p)$,
\end{itemize}
where $C(n,\hat g,p)$ is a positive constant independent of $t$.
\end{theorem}

\begin{remark}
In this paper, we assume that $T_0\le 1$ for convenience.
\end{remark}
\begin{remark}
The existence of $T_0$ and $g(t)$ and (1) have been proved by Simon (see Theorem \ref{thm2.2}). Shi-Tam \cite{ShTa} also got similar estimates as (2). Here we give a proof by using Moser's iteration. 
\end{remark}
 To prove Theorem \ref{mainsec6}. We consider $h$-flow (see \cite{Si02}). 
\begin{definition}
Given a constant $\delta\geq 0$, a metric $h$ is called to be $(1+\delta)$-fair to $g$, if $h$ is $C^\infty$, 
\[\sup_M|\tilde\nabla^j {\rm{Rm}}(h)|=k_j<\infty,\]
and \[(1+\delta)^{-1}h\leq g\leq (1+\delta) h \quad on\ M.\]
\end{definition}
Here and below, $\tilde \nabla$ means the covariant derivative taken with respect to $h$.

\begin{remark}\label{rmkh}
Let $M$ be a compact manifold and $g$ be a $C^0$ metric on $M$, then for any $0<\epsilon<1$, there exists a smooth metric $h$ which is $(1+\epsilon)$-fair to $g$. For a proof, see the remarks below \cite[Definition 1.1]{Si02}.
\end{remark}
\begin{definition}\label{dfn8}[$h$-flow] Given a background smooth metrics $h$, the $h$-flow is a family of metrics $g$ satisfies \[\frac{\partial}{\partial t}g_{ij}=-2R_{ij}+ \nabla_iV_j+ \nabla_jV_i,\]
where the derivatives are taken with respect to $g$, \[V_j=g_{jk}g^{pq}(\Gamma_{pq}^k-\tilde\Gamma_{pq}^k),\]
and $\Gamma$ and $\tilde\Gamma$ are the Christoffel symbols of $g$ and $h$ respectively. 
\end{definition}

$h$-flow is equivalent to the Ricci flow modulo an action of diffeomorphisms (see \cite{Si02}), thus we only need to prove Theorem \ref{mainsec6} for $h$-flow. Firstly we need the following theorem, which has been proved by Simon\cite{Si02}.

\begin{theorem} (\cite[Theorem 1.1]{Si02})\label{thm2.2}
There exists an $\epsilon(n)>0$ such that, for any compact $n$-manifold $M$ with a complete $C^0 $ metric $\hat g$ and a $C^\infty$ metric $h$ which is $(1+\frac{\epsilon(n)}{2})$-fair to $\hat g$, there exists a $T_0=T_0(n,k_0)>0$ and a family of metrics $g(t)\in C^{\infty}(M\times(0,T_0]),t\in (0,T_0]$ which solves $h$-flow for $t\in(0,T_0]$, $h$ is $(1+\epsilon(n))$-fair to $g(t)$, and
\begin{itemize}
\item[(1)] \begin{align*}
\lim_{t\to0^+}\sup_{x\in M}|g(x,t)-\hat g(x)|=0,
\end{align*}
\item[(2)] \begin{align*}
\sup_{M}|\tilde\nabla^ig(t)|\leq\frac{c_i(n,k_1,...,k_i)}{t^{i/2}}, \forall t\in(0,T_0],i\ge 1,\end{align*}
\end{itemize}
where the derivatives and the norms are taken with respect to $h$.

\end{theorem}

\begin{remark}
In order to apply the flow, we will let $h$ be $(1+\frac{\epsilon(n)}{2})$-fair to $\hat g$. By Remark \ref{rmkh}, such a metric always exists. 
\end{remark}

\begin{remark}\label{limitf}\cite{Si02}
Actually, take any family of smooth metrics $\{\hat g_\delta\}$ which converges to $\hat g$ uniformly on $M$ in $C^0$-norm, then $h$ is $1+\frac{\delta(n)}{2}$ fair to $\hat g_\delta$ for $\delta$ small enough. Starting from smooth metrics $\hat g_\delta$ we get $h$-flow $g_\delta(t)$, $t\in (0,T_0]$ with $T_0$ independent of $\delta$. Fix $t>0$ and let $\delta\to0^+$, by passing to a subsequence, we get $g(t)$, which is appeared to be the $h$-flow such that $g(0)=\hat g$ as in Theorem \ref{thm2.2}, and this convergence is a smooth convergence for each $t\in(0,T_0]$. For Ricci flow, the same procedure still works.
\end{remark}

When the initial metric is $W^{1,p}$, $n<p\le+\infty$, we have the following estimate: 
\begin{theorem}\label{thm6.2}
In the condition of Theorem \ref{thm2.2}, and moreover, if we assume that $M$ is compact and $\int_M|\tilde\nabla \hat g|^pd\mu_h<A$ for some constant $A$ and $n<p\le+\infty$, where the derivative and the norm are taken with respect to $h$. Then there exists a $T_0=T_0(n,h,A,p)$, such that $g(t)$, $ t\in(0,T_0]$ is the $h$-flow starting from metric $\hat g$, and
\begin{itemize}
\item[(1)]$\int_M|\tilde\nabla g(t)|^pd\mu_h\leq10A$, $\forall t\in(0,T_0]$,
\item[(2)]$|\tilde\nabla g|(t)\leq \frac{C(n,h,A,p)}{t^{\frac{n}{2p}}}$, $\forall t\in(0,T_0]$,
\item[(3)]$|\tilde\nabla^2 g|(t)\leq \frac{C(n,h,A,p)}{t^{\frac{n}{4p}+\frac{3}{4}}}$, $\forall t\in(0,T_0],$
\end{itemize}
where $C(n,h,A,p)$ is a positive constant independent of $t$.
\end{theorem}
To prove Theorem \ref{thm6.2}, we prove the following lemma at first.
\begin{lemma}\label{lm3}
In the condition of Theorem \ref{thm6.2}, if for some $T \in(0,1]$, there holds $\int_M|\tilde\nabla g(t)|^pd\mu_h\leq10A$, $\forall t\in(0,T ]$, then for the same $T$, $|\tilde\nabla g|(t)\leq \frac{C(n,h,A,p)}{t^{\frac{n}{2p}}}$, $\forall t\in(0,T ]$ also holds, where $C(n,h,A,p)$ is a positive constant independent of $t$.
\end{lemma}
\begin{proof}
By \cite{Shi1989}, we have the evolution equation of $|\tilde\nabla g_{ij}|^2$:
\begin{align*}
\frac{\partial}{\partial t}|\tilde\nabla g_{ij}|^2=&g^{\alpha\beta}\tilde\nabla_{\alpha}\tilde\nabla_{\beta}|\tilde\nabla g_{ij}|^2-2g^{\alpha\beta}\tilde\nabla_\alpha(\tilde\nabla g_{ij})\cdot\tilde\nabla_\beta(\tilde\nabla g_{ij})\\
&+\text{\~{Rm}}*g^{-1}*g^{-1}*g*\tilde\nabla g*\tilde\nabla g+g^{-1}*g*\tilde\nabla \text{\~{Rm}}*\tilde\nabla g\\
&+g^{-1}*g^{-1}*\tilde\nabla g*\tilde\nabla g*\tilde\nabla^2 g+g^{-1}*g^{-1}*g^{-1}*\tilde\nabla g*\tilde\nabla g*\tilde\nabla g*\tilde\nabla g,
\end{align*}
where the derivatives and norms are taken with respect to $h$, and \~{Rm} is the Riemannian curvature tensor of $h$.
Thus
\begin{align*}
\frac{\partial}{\partial t}|\tilde\nabla g_{ij}|^2-g^{\alpha\beta}\tilde\nabla_{\alpha}\tilde\nabla_{\beta}|\tilde\nabla g_{ij}|^2&\leq -C_1(n,h)|\tilde\nabla^2 g|^2\notag\\
&\quad+C(n,h)|\tilde\nabla g|^2+C(n,h)|\tilde\nabla g|\notag\\
&\quad+C(n)|\tilde\nabla g|^4+C_2(n)|\tilde\nabla g|^2|\tilde\nabla^2 g|\notag\\
&\leq -C_1(n,h)|\tilde\nabla^2 g|^2\notag\\
&\quad+C(n,h)|\tilde\nabla g|^2+C(n,h)|\tilde\nabla g|\notag\\
&\quad+C(n)|\tilde\nabla g|^4+\frac{C_2(n)}{2\epsilon}|\tilde\nabla g|^4+\frac{C_2(n)\epsilon}{2}|\tilde\nabla^2 g|^2, \forall \epsilon>0,
\end{align*}
here and below $C$ and $C_i$s are positive constants independent of $t$ and $C$ varies from line to line. Take $\epsilon=\frac{C_1}{C_2}$, we have
\begin{align}\label{ine5.1}
\frac{\partial}{\partial t}|\tilde\nabla g_{ij}|^2-g^{\alpha\beta}\tilde\nabla_{\alpha}\tilde\nabla_{\beta}|\tilde\nabla g_{ij}|^2&\leq -C(n,h)|\tilde\nabla^2 g|^2+C(n,h)(|\tilde\nabla g|+|\tilde\nabla g|^2+|\tilde\nabla g|^4)\\
&\leq C(n,h)(|\tilde\nabla g|+|\tilde\nabla g|^2+|\tilde\nabla g|^4).
\end{align}
Denote $f=|\tilde\nabla g_{ij}|^2+1$, then we have
\[\frac{\partial}{\partial t}f-g^{\alpha\beta}\tilde\nabla_\alpha\tilde\nabla_\beta f\leq C(n,h)(f+f^{\frac{1}{2}}+f^2)\leq C(n,h)f(1+f)\]
Denote $v=1+f$, then 
\begin{equation}\label{evof}\frac{\partial}{\partial t}f-g^{\alpha\beta}\tilde\nabla_\alpha\tilde\nabla_\beta f\leq C(n,h)fv.
\end{equation}
 Suppose that $T\in(0,1]$ is a constant such that $\int_M|\tilde\nabla g(t)|^pd\mu_h\leq10A$, $\forall t\in(0,T]$. Then $v$ has uniformly bounded $L^{\frac{p}{2}}{(M)}$-norm on $[0,T]$, in other words, we have
\[\int_Mv^{\frac{p}{2}}d\mu_h\leq C(n,h)A+C(n,h), \forall t\in[0,T].\]

For any $q> 0$, we multiple $  f^{q}$ to equation (\ref{evof}) and integrate it, then we get
\begin{equation}
\int_M\left(\frac{\partial}{\partial t}f-g^{\alpha\beta}\tilde\nabla_\alpha\tilde\nabla_\beta f\right) f^{q}d\mu_h\leq C(n,h)\int_Mf^{q+1}v d\mu_h.
\end{equation}
Thus we have
\begin{align}\label{a}
\frac{1}{q+1}\frac{\partial}{\partial t}\int_M f^{q+1}d\mu_h\leq
C(n,h)&\left[-\int_M(\tilde\nabla_\alpha g^{\alpha\beta})(\tilde\nabla_\beta f) f^{q}d\mu_h\notag\right.\\
&-\int_M g^{\alpha\beta}(\tilde\nabla_\beta f)(\tilde\nabla_\alpha  f^{q}) d\mu_h\notag\\
&\left.+\int_Mf^{q+1}v d\mu_h\right].
\end{align}
By H\"older inequality, we have
\begin{align}\label{b}
-\int_M(\tilde\nabla_\alpha g^{\alpha\beta})(\tilde\nabla_\beta f) f^{q}d\mu_h&\leq C(n,h)\int_Mf^{\frac{1}{2}}|\tilde\nabla f| f^{q}d\mu_h\notag\\
&\leq C_3\epsilon\int_M|\tilde\nabla f|^2 f^{q-1}d\mu_h\notag\\
&\quad +\frac{C_3}{\epsilon}\int_M f^{q+2}d\mu_h,\forall\epsilon>0.
\end{align}
And we have
\begin{align}\label{c}
\quad -\int_M g^{\alpha\beta} \tilde\nabla_\beta f \tilde\nabla_\alpha  f^{q} d\mu_h&=-q\int_M (g^{\alpha\beta}\tilde\nabla_\beta f\tilde\nabla_\alpha f)f^{q-1} d\mu_h\notag\\
&\leq -C_4 q\int_M|\tilde\nabla f|^2f^{q-1} d\mu_h.
\end{align}
Recall that $v=f+1$ and hence $f^{q+2}\leq f^{q+1}v$. Take $\epsilon=\frac{C_4 q}{2C_3}$ in (\ref{b}) and combine inequality (\ref{a}), (\ref{b}) and (\ref{c}), then we get
\begin{align}
\quad\frac{1}{q+1}\frac{\partial}{\partial t}\int_M f^{q+1}d\mu_h+C(n,h)q\int_M|\tilde\nabla f|^2f^{q-1} d\mu_h\leq C(n,h)\left(1+\frac{1}{q}\right)\int_Mf^{q+1}v d\mu_h.
\end{align}
Since
\begin{align*}
\quad\int_M\big|\tilde\nabla f^{\frac{q+1}{2}}  \big|^2d\mu_h
=\left(\frac{q+1}{2}\right)^2 \int_M\big|\tilde\nabla f\cdot f^{\frac{q-1}{2}} \big|^2d\mu_h,
\end{align*}
we get
\begin{align}\label{5.7}
&\quad\frac{1}{q+1}\frac{\partial}{\partial t}\int_M f^{q+1}d\mu_h+C(n,h)\frac{q}{(q+1)^2}\int_M\big|\tilde\nabla f^{\frac{q+1}{2}} \big|^2d\mu_h\notag\\
&\leq  C(n,h)\frac{q+1}{q}\int_Mf^{q+1}v d\mu_h.
\end{align}
If we let $q\geq\frac{p}{2}-1>0$, then $0<\frac{p-2}{p}\le \frac{q}{q+1}\leq1$,
thus we have
\begin{align}\label{8}
&\quad\frac{\partial}{\partial t}\int_M f^{q+1}d\mu_h+\int_M\big|\tilde\nabla f^{\frac{q+1}{2}} \big|^2d\mu_h \leq  C_5(n,h,p)q\int_Mf^{q+1}v d\mu_h.
\end{align}
For the last term, use H\"older inequality and since $p>n$ we can use interpolation inequality, then we get
\begin{align}\label{9}
\int_Mf^{q+1}v d\mu_h&\leq \left(\int_Mv^{\frac{p}{2}}d\mu_h\right)^{\frac{2}{p}}\left(\int_M(f^{q+1} )^{\frac{p}{p-2}}d\mu_h\right)^{\frac{p-2}{p}}\notag\\
&\leq C_6(n,h,A)\left[\epsilon\left(\int_M(f^{q+1} )^{\frac{n}{n-2}}d\mu_h\right)^{\frac{n-2}{n}}+\epsilon^{-\mu}\int_Mf^{q+1} d\mu_h\right],
\end{align}
where $\mu=(\frac{n-2}{n}-\frac{p-2}{p})/(\frac{p-2}{p}-1)=\frac{p-n}{n}$.
We also have the Sobolev inequality
\begin{equation}\label{10}\int_M\big|\tilde\nabla f^{\frac{q+1}{2}} \big|^2d\mu_h\geq C_7(n,h)\left(\int_M(f^{q+1} )^{\frac{n}{n-2}}d\mu_h\right)^{\frac{n-2}{n}}.
\end{equation}
Take $\epsilon=\frac{C_7}{2C_5C_6q}$ and combine (\ref{8}), (\ref{9}) and (\ref{10}), then we get
\begin{align}\label{15}
\frac{\partial}{\partial t}\int_Mf^{q+1}d\mu_h+\left(\int_M(f^{q+1})^{\frac{n}{n-2}}d\mu_h\right)^{\frac{n-2}{n}}\notag\\
\leq C(n,h,A,p)q^{1+\mu} \int_Mf^{q+1}d\mu_h.
\end{align}
For any $0<t'<t''<T'\leq T\leq1$, let
\[\psi(t)=\left\{
\begin{array}{rl}
0, & \text{if } 0\leq t\leq t',\\
\frac{t-t'}{t''-t'}, & \text{if } t'\leq t\leq t'',\\
1, & \text{if } t''\leq t\leq T.
\end{array} \right. \]
Multiplying (\ref{15}) by $\psi$, we get
\begin{align}\label{16}
\frac{\partial}{\partial t}\int_M\psi f^{q+1}d\mu_h+\psi\left(\int_M(f^{q+1})^{\frac{n}{n-2}}d\mu_h\right)^{\frac{n-2}{n}}\notag\\
\leq[ C(n,h,A,p)q^{1+\mu} \psi+\psi']\int_Mf^{q+1}d\mu_h.
\end{align}
Integrating it with respect to $t$, we get
\begin{align}
\sup_{t\in[t'',T']} \int_Mf^{q+1}d\mu_h+\int_{t''}^{T'}\left(\int_M(f^{q+1})^{\frac{n}{n-2}}d\mu_h\right)^{\frac{n-2}{n}}dt\notag\\
\leq[ C(n,h,A,p)q^{1+\mu} +\frac{1}{t''-t'}]\int_{t'}^{T'}\int_Mf^{q+1}d\mu_hdt.
\end{align}
Then we have
\begin{align}\label{14}
&\quad\int_{t''}^{T'}\int_Mf^{(q+1)(1+\frac{2}{n})}d\mu_hdt\notag\\
&\leq \int_{t''}^{T'}\left(\int_Mf^{q+1}d\mu_h\right)^{\frac{2}{n}}\left(\int_Mf^{(q+1)\frac{n}{n-2}}d\mu_h\right)^{\frac{n-2}{n}}dt\notag\\
&\leq\sup_{t\in[t'',{T'}]}\left(\int_Mf^{q+1}d\mu_h\right)^{\frac{2}{n}}\int_{t''}^{T'}\left(\int_Mf^{(q+1)\frac{n}{n-2}}d\mu_h\right)^{\frac{n-2}{n}}dt\notag\\
&\leq[ C(n,h,A,p)q^{1+\mu} +\frac{1}{t''-t'}]^{1+\frac{2}{n}}\left(\int_{t'}^{T'}\int_Mf^{q+1}d\mu_hdt\right)^{1+\frac{2}{n}},\forall q\ge\frac{p}{2}-1.
\end{align}
Denote \[H(q, \tau)=\left(\int_{\tau}^{T'}\int_{M}f^{q}d\mu_hdt\right)^{\frac{1}{q}},\forall q\geq\frac{p}{2},  0<\tau<T'.\]
Then equation (\ref{14}) can be shortly writen as
\[H(q(1+\frac{2}{n}), t'')\leq[ {C(n,h,A,p)}q^{1+\mu}+\frac{1}{t''-t'}]^\frac{1}{q} H(q, t').\]
Fix $0<t_0<t_1<T'\leq1$, $q_0\geq\frac{p}{2}$ and set $\chi=1+\frac{2}{n}$, $q_k=q_0\chi^k$, $\tau_k=t_0+(1-\frac{1}{\chi^k})(t_1-t_0)$. Then we have (see also \cite{Jiang}, \cite{JWZ})
\begin{align*}H(q_{k+1}, \tau_{k+1})&\leq[ C(n,h,A,p) q_k^{1+\mu}+\frac{1}{t_1-t_0}\frac{\chi}{\chi-1}\chi^k]^\frac{1}{q_k} H(q_k, \tau_k)\\
&= [C(n,h,A,p)q_0^{1+\mu}\chi^{k(1+\mu)}+\frac{1}{t_1-t_0}\frac{n+2}{2}\chi^k]^\frac{1}{q_k}H(q_k, \tau_k)\\
&\leq [\frac{C(n,h,A,p,q_0)}{t_1-t_0}]^\frac{1}{q_k}\chi^{\frac{k(1+\mu)}{q_k}} H(q_k, \tau_k),
\end{align*}
where in the last inequality we use $0<t_0<t_1<T$.
By iteration, we get
\begin{align*}
H(q_{m+1}, \tau_{m+1})
&\leq [\frac{C(n,h,A,p,q_0)}{t_1-t_0}]^{\sum_{k=0}^m\frac{1}{q_k}}\chi^{\sum_{k=0}^m\frac{k(1+\mu)}{q_k}} H(q_0, \tau_0)\\
&\leq C(n,h,A,p,q_0)(\frac{1}{t_1-t_0})^{\frac{n+2}{2q_0}}H(q_0, \tau_0),
\end{align*}
since $\sum_{k=0}^\infty\frac{1}{q_k}=\frac{n+2}{2q_0}$ and $\sum_{k=0}^\infty\frac{k(1+\mu)}{q_k}$ converges.
Letting $m\to\infty$, we get
\[H(p_\infty, \tau_\infty)\leq C(n,h,A,p,q_0)(\frac{1}{t_1-t_0})^{\frac{n+2}{2q_0}}H(q_0, \tau_0), \forall q_0\geq\frac{p}{2},\]
where $p_\infty=+\infty,\tau_\infty=t_1$.
Letting $q_0=\frac{p}{2}$, we have
\[H(\infty, t_1)\leq C(n,h,A,p)(\frac{1}{t_1-t_0})^{\frac{n+2}{p}}H(\frac{p}{2}, t_0).\]
Thus we have
\[\sup_{(y,t)\in M\times[t_1,T']}f(y,t)\leq C(n,h,A,p)\frac{1}{(t_1-t_0)^{\frac{n+2}{p}}}\left(\int_{t_0}^{T'}A^{\frac{p}{2}}dt\right)^{\frac{2}{p}}.\]
Letting $t_1\to T'$ and $t_0=T'/2$, we get
\[\sup_{y\in M} f(y,T')\leq C(n,h,A,p)\frac{1}{(T')^{\frac{n}{p}}},\forall T'\in(0,T].\]
Thus we have
\[|\tilde\nabla g|(t)\leq \frac{C(n,h,A,p)}{t^{\frac{n}{2p}}}, \forall t\in(0,T],\]
which completes the proof of the lemma.
\end{proof}

\begin{proof}[Proof of Theorem \ref{thm6.2}]
The existence of the $h$-flow $g(t)$ is claimed in Theorem \ref{thm2.2}, so we only need to prove that conclusions (1), (2) and (3) hold for some $T_0(n,h,A,p)$.

To prove that (1) and (2) hold for some $T_0(n,h,A,p)$, let $f=|\tilde\nabla g_{ij}|^2+1$.\\
We denote 
\[\phi(t)= \int_Mf^{\frac{p}{2}} d\mu_h.\]
Then we have
\begin{align}
\phi'(t)=\int_M\frac{\partial}{\partial t}(f^{\frac{p}{2}} ) d\mu_h.
\end{align}

We denote 
\[\mathcal{T}=\bigg\{T\in(0,1]\big|\int_M|\tilde\nabla g(t)|^pd\mu_h\leq10A,\forall x\in M, \forall t\in[0,T]\bigg\},\]
and $T_{\max}=\sup\mathcal{T}$.
Take $q=p/2-1$, then equation (\ref{8}) gives
\[\phi'(t)\leq C(n,h,p)\int_Mf^{p/2}v d\mu_h,\]
where $v=f+1$.
By Lemma \ref{lm3}, we have 
\begin{align*}
v=f+1\le 1+\frac{C(n,h,A,p)}{t^{n/p}}, \forall t\in(0,T_{\max}].
\end{align*}
Thus we get 
\begin{equation}
\phi'(t)\leq \left(1+\frac{C(n,h,A,p)}{t^{n/p}}\right)\int_Mf^{p/2} d\mu_h, \forall t\in(0,T_{\max}].
\end{equation}
Since $T_{\max}\leq 1$,
we have
\[\phi'(t)\leq \frac{C(n,h,A,p)}{t^{n/p}}\phi(t), \forall t\in(0,T_{\max}],\]
thus we have
\[(\log\phi(t))' \leq\frac {C(n,h,A,p)}{t^{n/p}}, \forall t\in(0,T_{\max}].\]
By integration, we get
\[\log\phi(t)\leq \log\phi(0)+C(n,h,A,p)t^{1-n/p}, \forall t\in(0,T_{\max}],\]
thus we have
\[\phi(t)\leq\phi(0)e^{C(n,h,A,p)t^{1-n/p}}, \forall t\in(0,T_{\max}],\]
Since $\phi(0)\leq C(n,h)A$, we have
\begin{equation}\label{21}\int_M|\tilde\nabla g(t)|^pd\mu_h\leq \phi(t)\leq C(n,h)e^{C(n,h,A,p)t^{1-n/p}}A, \forall t\in(0,T_{\max}].
\end{equation}
Suppose that if we fix $n$, $h$, $p$ and $A$, there exists a sequence of initial metric $\{\hat g_m\}_{m=1}^\infty$, such that each of the metric $\hat g_m$ satisfies the condition of the theorem and the corresponding $T_{\max;m}$ tends to $0$. But by (\ref{21}) we know that if $T_{\max}$ satisfies $C(n,h)e^{C(n,h,A,p)T_{\max}^{1-n/p}}\leq 5$ and $T_{\max}<1$, then the maximal time interval $(0,T_{\max}]$ could be extended to $(0,T_{\max}+\delta]$, for some $\delta>0$ small enough, which is a contradiction. Hence $T_{\max}$ is only depend on $n$, $h$, $p$ and $A$. Thus we get that the conclusion (1) and (2) hold simultaneously for some $T_0(n,h,A,p)$.

To prove (3), recall that Simon's result, Theorem \ref{thm2.2}, gives
\[\sup_{M}|\tilde\nabla^ig(t)|\leq\frac{c_i(n,h)}{t^{i/2}},\quad\forall t\in(0,T_0].\]
We choose a finite atlas for $M$ such that $h$ is uniformly equivalent to the Euclidean metric of
each chart. For any chart $(U,\Phi)$, we let $\bar f$ denote any component function $\tilde\nabla_i g_{jk}(t)$ of $\tilde\nabla g(t)$, and let $h_0$ denote the Euclidean metric of $(U,\Phi)$. We can assume that $2^{-1}h_0\le h \le 2 h_0$. We choose an arbitrary point $p\in U$ and let $\gamma(u)$ be a curve satisfying $\gamma(0)=p$ and $\gamma ' (u)\equiv\frac{\partial}{\partial x^i}$, where $\frac{\partial}{\partial x^i}$ is the coordinate vector field of $(\Phi, U)$. Then we have
\begin{align*}
\int_0^s \frac{d}{du} \bar f\comp\gamma(u) du=\bar f\comp\gamma(s)-\bar f(p),\forall s\in(0,s_0),
\end{align*}
where $s_0$ is a positive constant independent of the chart and $t$, and $s_0$ is small enough such that $\gamma(u)\in U$ for any $u\in(0,s_0)$ and for any coordinate neighborhood $U$.\\
On the one hand, we have 
\begin{align*}
\int_0^s \frac{d}{du} \bar f\comp\gamma(u) du=\int_0^s d\bar f|_{ \gamma(u) } (\gamma '(u))du= \int_0^s  \frac{\partial \bar f}{\partial x^i}(\gamma(u)) du.
\end{align*}
On the other hand, we have
\begin{align*}
\bar f\comp\gamma(s)-\bar f(p)\le  2\sup_M |\tilde\nabla g(t)|\le \frac{C(n,h,A,p)}{t^{\frac{n}{2p}}}.
\end{align*} 
Thus there exists $u_0\in(0,s)$, such that 
\begin{align*}
\frac{\partial \bar f}{\partial x^i}(\gamma(u_0))\le \frac{C(n,h,A,p)}{st^{\frac{n}{2p}}}.
\end{align*}
Since 
\begin{align*}\frac{\partial \bar f}{\partial x^i}(\gamma(u_0))-\frac{\partial \bar f}{\partial x^i}(\gamma(0))&=\int_0^{u_0}\frac{\partial^2 \bar f}{(\partial x^i)^2}(\gamma(u))du\\
&\ge -u_0 C(n,h)\sup_M|\tilde\nabla^3 g(t)|\\
&\ge -s \frac{C(n,h)}{t^{\frac{3}{2}}}.
\end{align*}
We have 
\begin{align}\label{i1}
\frac{\partial \bar f}{\partial x^i}(p)=\frac{\partial \bar f}{\partial x^i}(\gamma(0))
&\le \frac{\partial \bar f}{\partial x^i}(\gamma(u_0))+s \frac{C(n,h)}{t^{\frac{3}{2}}}\notag\\
&\le \frac{C(n,h,A,p)}{st^{\frac{n}{2p}}}+s \frac{C(n,h)}{t^{\frac{3}{2}}}.
\end{align}
The right hand side of (\ref{i1}) attains its minimum for $s=C(n,h,A,p)t^{\frac{3}{4}-\frac{n}{4p}}$. Note that $\lim_{t\to 0^+}t^{\frac{3}{4}-\frac{n}{4p}}=0$, thus there exists $T_0=T_0(n,h,A,p)$, such that for any $t\in(0,T_0)$, we have $C(n,h,A,p)t^{\frac{3}{4}-\frac{n}{4p}}\le s_0$.
Therefore, we have that for any $t\in(0,T_0)$, the right hand side of (\ref{i1}) attains its minimum for some $s\in(0,s_0)$, and the minimum value is $\frac{C(n,h,A,p)}{t^{\frac{3}{4}+\frac{n}{4p}}}$.

Since $p$ is an arbitrary point on $M$ and $2^{-1}h_0\le h \le 2 h_0$, we get
\[|\tilde\nabla^2 g|(t) 
\leq \frac{C(n,h,p,A)}{t^{\frac{n}{4p}+\frac{3}{4}}}.\]
Thus (3) holds for some $T_0(n,h,A,p)$, which  completes the proof of the theorem.
\end{proof}

Theorem \ref{mainsec6} (2) follows immediately from Theorem \ref{thm6.2}. Moreover, for Theorem \ref{mainsec6} (3), we just need to prove the following lemma:
\begin{lemma}\label{lmR}
In the condition of Theorem \ref{thm6.2}, we have
\begin{align*}
\int_0^{T_0}\int_M |\tilde \nabla^2 g(t)|^2d\mu_hdt\le C(n,h,A,p).
\end{align*}
\end{lemma}
\begin{proof}
We calculate as the proof of Lemma \ref{lm3}, but this time we preserve the $|\tilde\nabla^2 g|^2$ term in (\ref{ine5.1}). Integrate both sides of the inequality (\ref{evof}), then we have
\begin{align}
\frac{\partial}{\partial t}\int_M fd\mu_h+\int_M |\tilde\nabla^2 g|^2d\mu_h\notag
&\leq C(n,h)\int_M g^{\alpha\beta}\tilde\nabla_\alpha\tilde\nabla_\beta fd\mu_h+ C(n,h)\int_Mfv d\mu_h,
\end{align}
where $f=|\tilde\nabla g_{ij}|^2+1$ and $v=f+1$, here and below $C$ and $C_i$s denote a positive constant and $C$ varies from line to line. Using integration by parts twice, we have 
\begin{align*}
\frac{\partial}{\partial t}\int_M fd\mu_h+\int_M |\tilde\nabla^2 g|^2d\mu_h\notag
&\leq C(n,h)\int_M \tilde\nabla_\beta \tilde\nabla_\alpha g^{\alpha\beta} fd\mu_h+ C(n,h)\int_Mfv d\mu_h\\ 
&\le C_1(n,h)\int_M |\tilde\nabla^2 g| fd\mu_h+ C(n,h)\int_Mfv d\mu_h\\ 
&\le C_1(n,h)\epsilon \int_M |\tilde\nabla^2 g|^2 d\mu_h+C_1(n,h)\epsilon^{-1}\int_M f^2d\mu_h+ C(n,h)\int_Mfv d\mu_h,\forall\epsilon>0. 
\end{align*}
Taking   $\epsilon=\frac{C_1}{2}$, since $f\le v$, we have
\begin{align}\label{5.21}
 \frac{\partial}{\partial t}\int_M f d\mu_h +\int_M |\tilde\nabla^2 g|^2d\mu_h \leq  C_2(n,h) \int_Mf v d\mu_h.
\end{align}
For the last term, using H\"older inequality and interpolation inequality, we get
\begin{align}\label{5.22}
\int_Mf v d\mu_h&\leq \left(\int_Mv^{\frac{p}{2}}d\mu_h\right)^{\frac{2}{p}}\left(\int_M f^{\frac{p}{p-2}}d\mu_h\right)^{\frac{p-2}{p}}\notag\\
&\leq C_3(n,h,A)\left[\epsilon\left(\int_M f ^{\frac{n}{n-2}}d\mu_h\right)^{\frac{n-2}{n}}+\epsilon^{-\mu}\int_M f d\mu_h\right],
\end{align}
where $\mu=(\frac{n-2}{n}-\frac{p-2}{p})/(\frac{p-2}{p}-1)=\frac{p-n}{n}$.
We also have the Sobolev inequality
\begin{equation}\label{5.23} \left(\int_M f ^{\frac{n}{n-2}}d\mu_h\right)^{\frac{n-2}{n}} \le C_4(n,h)\int_M\left|\tilde\nabla f^{\frac{ 1}{2}} \right|^2d\mu_h.
\end{equation}
Since 
\begin{align}\label{5.24}
\tilde\nabla f^{\frac{1}{2}}
=\tilde\nabla(|\tilde\nabla g|^2+1)^\frac{1}{2}\notag
=\frac{\tilde\nabla \langle \tilde\nabla g,\tilde\nabla g\rangle}{2(|\tilde\nabla g|^2+1)^\frac{1}{2}}
&=\frac{2\langle \tilde\nabla^2 g,\tilde\nabla g\rangle}{2(|\tilde\nabla g|^2+1)^\frac{1}{2}}\notag\\
&\le \frac{|\tilde\nabla^2 g||\tilde\nabla g|}{(|\tilde\nabla g|^2+1)^\frac{1}{2}} \le |\tilde\nabla^2 g|,
\end{align}
Taking $\epsilon=\frac{1}{2C_2C_3C_4}$, by (\ref{5.21}), (\ref{5.22}), (\ref{5.23}) and (\ref{5.24}), we have
 \begin{align*} 
 \frac{\partial}{\partial t}\int_M f d\mu_h +\int_M |\tilde\nabla^2 g|^2d\mu_h \leq  C(n,h,A,p) \int_Mf  d\mu_h.
\end{align*}
By Theorem \ref{thm6.2}, $\int_M|\tilde\nabla g(t)|^pd\mu_h\leq10A$, $\forall t\in(0,T_0]$. Since $p>n\ge 2$, we have $\int_Mf  d\mu_h\le C(n,h,A,p)$, thus we have
 \begin{align*} 
 \frac{\partial}{\partial t}\int_M f d\mu_h +\int_M |\tilde\nabla^2 g|^2d\mu_h \leq  C(n,h,A,p).
\end{align*}
Integrate it, we get
\begin{align*}
\int_M f(T_0)d\mu_h-\int_M f(0)d\mu_h+\int_0^{T_0}\int_M |\tilde\nabla^2 g|^2d\mu_h dt \leq  C(n,h,A,p).
\end{align*}
Since $\int_M f(T_0)d\mu_h\ge 0$ and $\int_M f(0)d\mu_h \le C(n,h,A,p)$, we get 
\begin{equation*}
\int_0^{T_0}\int_M |\tilde\nabla^2 g|^2d\mu_h dt \leq  C(n,h,A,p),
\end{equation*}
which completes the proof of the lemma.
\end{proof}
Now Theorem \ref{mainsec6} would follow without much effort.
\begin{proof}[Proof of Theorem \ref{mainsec6}]
Let $g(t)$ be the $h$-flow stated in Theorem \ref{thm2.2}. Then there is a family of diffeomorphisms $\phi(t):M \rightarrow M$, such that $\phi(t)^* g(t),t\in(0,T_0]$ is a Ricci flow. (See \cite{Si02}).

Since $M$ is compact, by Theorem \ref{thm2.2} (1) we have that $(M,g(t))$ converges to $(M,\hat g)$ in Gromov-Hausdorff distance. Since the Ricci flow $\phi(t)^* g(t)$ and the $h$-flow $g(t)$ are isometric for each $t\in (0,T_0]$, Theorem \ref{mainsec6} (1) holds.

To prove (2) and (3),
note that ${\rm{Rm}=\partial^2 g+\partial g*\partial g}$ and $p>n\ge 2$. Since the Ricci flow $\phi(t)^* g(t)$ and the $h$-flow $g(t)$ are  isometric for each $t\in(0,T_0]$ and both $g(t)$ and $\phi(t)^* g(t)$ are uniformly equivalent to $h$, Theorem \ref{mainsec6} (2) and (3) follows immediately from Theorem \ref{thm6.2} and Lemma \ref{lmR}.
\end{proof}

\section{Conjugate heat equation}
Let $(M^n,\hat g)$ be a compact manifold let $g(t)$ be the Ricci flow starting from the metric $\hat g$. In this section we suppose that $\hat g$ is smooth. Let $\mu_{g(t)}$ be the Lebesgue measure taken with respect to $g(t)$, and let $R_{\hat g} $ and $R_{g(t)} $ denote the distirbutional scalar curvature of $\hat g$ and $g(t)$ respectively. Let $\tilde \varphi$ be an arbitrary nonnegative function in $C^\infty(M)$ and take any $T\in (0,T_0]$. We consider the following conjugate heat equation
\begin{align}\label{che}\left\{
\begin{array}{rl}
&\partial_t\varphi_t=-\Delta_{g(t)}\varphi_t+R_{g(t)} \varphi_t \quad on\ M\times[0,T],\\
&\left.\varphi_t\right|_{t=T}=\tilde\varphi,
\end{array} \right.
\end{align}
where $\Delta_{g(t)}$ is the Laplacian taken with respect to $g(t)$.

For fixed $(x,t)\in M\times (0,T_0]$, the conjugate heat kernel on the Ricci flow background is the function $K(x,t;\cdot,\cdot)$, defined for $0\le s<t$ and $y\in M$ and satisfying
\begin{align*}
(-\partial_s-\Delta_y +R(y,s))K(x,t;y,s)=0\qquad and \qquad \lim_{s\to t^-}K(x,t;y,s)=\delta_x(y),
\end{align*}
where the Laplacian is taken with respect to $g(s)$, and $\delta_x$ is the Dirac operator supported on $\{x\}$. $K$ also satisfies $(\partial_t-\Delta_x )K(x,t;y,s)=0$, where $\Delta_x$ is taken with respect to $g(t)$.

By directly calculation(see also \cite{BaZh}), we get that equation (\ref{che}) has a solution with the explicity expression
\begin{equation}\label{expphi}
\varphi_t(x)=\int_M K(y,T;x,t)\tilde \varphi(y)d\mu_{g(T)}(y).
\end{equation}
By the maximum principle, we get that this solution is nonnegative and unique. Furthermore, by this expression we can see that $\varphi_t$ is uniformly bounded. Our main purpose in this section is proving the following estimates for $\varphi_t$ which will be used in the proof of our main theorem.
\begin{proposition}\label{prop_varphit}
	Assume as above, then $\varphi_t$ satisfies
\begin{itemize}
	\item[(1)] $\varphi_t\le C(n,h,A,p,\|\tilde\varphi\|_{L^\infty})$, $\forall t\in[0,T].$
	\item[(2)] $\int_M | \nabla_{g(t)} \varphi_t|^2_{g(t)}d\mu_{g(t)}\le C(n,h,A,p,\tilde\varphi)$,  $\forall t\in[0,T].$
	\item[(3)] $\int_M (R_{g(t)} -a) \varphi_td\mu_{g(t)}$ is monotonously increasing with respect to $t$. 
\end{itemize}
\end{proposition}
\begin{proof}
	To see (1), by equation (\ref{expphi}), we have
\begin{align}\label{phi64}
\varphi_t(x)\leq \|\tilde\varphi\|_{L^\infty}\int_M K(y,T;x,t)d\mu_{g(T)} (y).
\end{align}
We denote $F(t,T)=\int_M  K(y,T;x,t)d\mu_{g(T)} (y)$, then we have $\lim_{T\to t^+}F(t,T)=1$, and
\begin{align*}
\partial_TF(t,T)=\int_M\left(\Delta_y K(y,T;x,t)-R_T K(y,T;x,t)\right)d\mu_{g(T)} (y),
\end{align*}
where we have used the standard evolution equation $\partial_T d\mu_{g(T)} =-R_Td\mu_{g(T)} $. By divergence theorem, we have 
\begin{align*}
\int_M \Delta_y K(y,T;x,t) d\mu_{g(T)} (y)=0.
\end{align*}
Thus by Theorem \ref{mainsec6}, we have 
\begin{align*}
\partial_T F(t,T)\le \frac{C(n,h,A,p)}{T^\alpha}F(t,T),
\end{align*}
for some $\alpha\in(0,1)$.
Since $\lim_{T\to t^+}F(t,T)=1$, by integration we have
\begin{align}\label{phi65}
F(t,T)\le C(n,h,A,p),\forall 0\le t<T\le T_0.
\end{align}
By (\ref{phi64}) and (\ref{phi65}), we have
\begin{align*}
\varphi_t\le C(n,h,A,p,\|\tilde\varphi\|_{L^\infty}),\forall t\in[0,T],
\end{align*}
which proves (1).

To prove (2), denote $E(t)=\int_M | \nabla_{g(t)} \varphi_t|^2_{g(t)}d\mu_{g(t)} $. By direct calculation, we have
\begin{align}\label{a6.6}
\partial_t E(t)=\int_M \left(-R_{g(t)} | d \varphi_t|_{g(t)}^2+2\Ric_{g(t)}( \nabla_{g(t)} \varphi_t,  \nabla_{g(t)} \varphi_t)+2 \langle   \nabla_{g(t)} \partial_t \varphi_t,  \nabla_{g(t)} \varphi_t\rangle_{g(t)}\right)d\mu_{g(t)}.
\end{align}
By (\ref{che}) and using Bochner formula we have
\begin{align}\label{a67}
\int_M \langle   \nabla_{g(t)} \partial_t \varphi_t,  \nabla_{g(t)} \varphi_t\rangle_{g(t)} d\mu_{g(t)}&=\int_M -\langle   \nabla_{g(t)} (\Delta_{g(t)} \varphi_t-R_{g(t)}  \varphi_t),  \nabla_{g(t)} \varphi_t\rangle_{g(t)} d\mu_{g(t)} \notag\\ 
&=\int_M \left(-\langle \nabla_{g(t)} \Delta_{g(t)} \varphi_t ,\nabla_{g(t)} \varphi_t\rangle_{g(t)} +\langle\nabla_{g(t)}(R_{g(t)} \varphi_t),\nabla_{g(t)}\varphi_t\rangle_{g(t)}\right)d\mu_{g(t)}\notag \\ 
&=\int_M \left(-\frac{1}{2}\Delta_{g(t)}|\nabla_{g(t)}\varphi_t|_{g(t)}^2+|\nabla_{g(t)}^2\varphi_t|_{g(t)}^2+\Ric_{g(t)}(\nabla_{g(t)}\varphi_t,\nabla_{g(t)}\varphi_t)\right.\notag \\ 
&\left.\qquad\qquad+\langle\nabla_{g(t)}(R_{g(t)} \varphi_t),\nabla_{g(t)}\varphi_t\rangle_{g(t)}\right)d\mu_{g(t)}\notag \\ 
&=\int_M \left[\left(|\nabla_{g(t)}^2\varphi_t|_{g(t)}^2+\Ric_{g(t)}(\nabla_{g(t)}\varphi_t,\nabla_{g(t)}\varphi_t)\right)-R_{g(t)} \varphi_t\Delta_{g(t)}\varphi_t\right]d\mu_{g(t)}.
\end{align}
Since $|\Delta_{g(t)}\varphi_t|_{g(t)}^2\le C(n)|\nabla_{g(t)}^2\varphi_t|^2$, using Cauchy inequality, (\ref{a67}) gives
\begin{align}\label{a6.8}
 \int_M \langle   \nabla_{g(t)} \partial_t \varphi_t,  \nabla_{g(t)} \varphi_t\rangle_{g(t)} d\mu_{g(t)}\ge \int_M \left(\Ric_{g(t)}(\nabla_{g(t)}\varphi_t,\nabla_{g(t)}\varphi_t)-C(n)R_{g(t)} ^2\varphi_t^2\right)d\mu_{g(t)}.
\end{align}
By (\ref{a6.6}) and (\ref{a6.8}), since $|R_{g(t)} |_{g(t)}\le c(n)|\Ric_{g(t)}|_{g(t)}$, we get
\begin{align}
\partial_tE(t)&\ge \int_M \left(4\Ric_{g(t)}(\nabla_{g(t)}\varphi_t,\nabla_{g(t)}\varphi_t)-R_{g(t)}|\nabla_{g(t)} \varphi_t|_{g(t)}^2 -C(n)R_{g(t)}^2\varphi_t^2\right)d\mu_{g(t)}\notag \\ 
&\ge -C(n,h,A,p)\int_M |\Ric_{g(t)}||\nabla_{g(t)}\varphi_t|^2d\mu_{g(t)}-C(n,h,A,p)\int_M R_{g(t)} ^2\varphi_t^2d\mu_{g(t)}.
\end{align}
By Theorem \ref{thm66.2} and (1), we have for some $\alpha\in(0,1)$ that
\begin{align*}
\partial_t E(t)\ge -\frac{C(n,h,A,p)}{t^\alpha}\int_M |\nabla_{g(t)} \varphi_t|^2d\mu_{g(t)}-C(n,h,A,p,\tilde\varphi)\int_M R_{g(t)} ^2d\mu_{g(t)}.
\end{align*}
 Thus we have
\begin{align*}
\partial_t\left(E(t)+1\right)\ge -\frac{C(n,h,A,p)}{t^\alpha}\left(E(t)+1\right)-C(n,h,A,p,\tilde\varphi)\int_M R_{g(t)} ^2d\mu_{g(t)}.
\end{align*}
Dividing both sides by $\left(E(t)+1\right)$, we get
\begin{align*}
\partial_t \log\left(E(t)+1\right)\ge -\frac{C(n,h,A,p)}{t^\alpha}-C(n,h,A,p,\tilde\varphi)\int_M R_{g(t)} ^2d\mu_{g(t)}.
\end{align*}
 By Lemma \ref{lmR}, $\int_M R_{g(t)} ^2d\mu_{g(t)}$ is integrable on $(0,T)$. Since $\varphi_T=\tilde\varphi$, $E(T)\le C(n,h,A,p,\tilde\varphi)$ and $\frac{1}{t^\alpha}$ is integrable on $(0,T)$, we can integrate the inequality above and get
\[E(t)\le C(n,h,A,p,\tilde\varphi), \forall t\in[0,T].
\]
Thus we complete the proof of (2).

To prove (3), we can directly calculate that
\begin{align*}
&\partial_t \int_M (R_{g(t)} -a)\varphi_t d\mu_{g(t)}\\
=& \int_M \left[\left(\Delta_{g(t)}  R_{g(t)} +2|\Ric_{g(t)} |_{g(t)}^2\right)\varphi_t + (R_{g(t)} -a)\left(-\Delta_{g(t)} \varphi_t+ R_{g(t)} \varphi_t\right) + (R_{g(t)} -a)\varphi_t\left(-R_{g(t)} \right)\right]d\mu_{g(t)}\\
=& \int_M 2|\Ric_{g(t)} |_{g(t)}^2\varphi_t d\mu_{g(t)} + \int_M \left(\Delta_{g(t)}  R_{g(t)}  \varphi_t - R_{g(t)}  \Delta_{g(t)}  \varphi_t \right)d\mu_{g(t)}+a\int_M\Delta_{g(t)} \varphi_td\mu_{g(t)}.
\end{align*}
By integration by parts, we have that 
\begin{equation*}
\int_M \left(\Delta_{g(t)}  R_{g(t)}  \varphi_t - R_{g(t)}  \Delta_{g(t)}  \varphi_t \right)d\mu_{g(t)}=0,\int_M\Delta_{g(t)} \varphi_td\mu_{g(t)}=0.
\end{equation*}
Thus we have 
\begin{equation*}
\partial_t \int_M (R_{g(t)} -a)\varphi_t d\mu_{g(t)} \ge 0.
\end{equation*}
Therefore $\int_M (R_{g(t)} -a) \varphi_td\mu_{g(t)}$ is monotonously increasing. Hence we finish the proof of (3) and thus the proof of the proposition.
\end{proof}

\section{Proof of the main theorem}
In this section, we will give the proof of our main theorem. Let us restate Theorem \ref{thm1.1} as follows. 
\begin{theorem}\label{mthm2}
Let $M^n$ be a compact manifold with a metric $\hat g\in W^{1,p}(M)$ $(n< p\le \infty)$. Assume the distirbutional scalar curvature $R_{\hat g}\ge a$ for some constant $a$, and let $g(t),t\in(0,T_0]$ be the Ricci flow starting from the metric $\hat g$. Then for any $t\in(0,T_0]$, there holds $R_{g(t)}\ge a$ on $M$. 
\end{theorem}
\begin{proof}
By Lemma \ref{lm2.1}, we get a family of smooth metrics $\hat g_\delta$ which converges to $\hat g$ both in $C^0$-norm and $W^{1,p}$-norm. Then by Lemma \ref{lm2.2}, we have
\begin{align}\label{d1}
\left|\int_M R_{\hat g_\delta}\varphi d\mu_{\hat g_\delta}-\langle R_{\hat g},\varphi\rangle\right|\le b'_\delta \|\varphi\|_{W^{1,\frac{n}{n-1}}(M)}, \forall \varphi\in C^\infty(M),\forall\delta\in(0,\delta_0],
\end{align}
where $b'_\delta$ is a positive function of $\delta$ which only depends on $\hat g$ and satisfies $\lim_{\delta\to0^+}b'_\delta=0$ and $\delta_0$ is some positive constant small enough.

Moreover, we have $\lim_{\delta\to0^+}\left\|\frac{d\mu_{\hat g_\delta}}{d\mu_{\hat g}}-1\right\|_{C^0(M)}=0$, thus by H\"older inequality, we have
\begin{align}\label{d2}
\left|\int_M\varphi d\mu_{\hat g_\delta}-\int_M\varphi d\mu_{\hat g}\right|
&= \left|\int_M\varphi \left(\frac{d\mu_{\hat g_\delta}}{d\mu_{\hat g}}-1 \right)d\mu_{\hat g}\right|\notag\\
&\le \left\|\frac{d\mu_{\hat g_\delta}}{d\mu_{\hat g}}-1\right\|_{C^0(M)}\int_M|\varphi| d\mu_{\hat g}
\notag\\
&\le  C(n,\hat g)\left\|\frac{d\mu_{\hat g_\delta}}{d\mu_{\hat g}}-1\right\|_{C^0(M)}\|\varphi\|_{W^{1,\frac{n}{n-1}}(M)}.
\end{align}
By triangular inequality, combining (\ref{d1}) and (\ref{d2}), we have
\begin{align}\label{d3}
\left|\int_M (R_{\hat g_\delta}-a)\varphi d\mu_{\hat g_\delta}-\langle R_{\hat g}-a,\varphi\rangle\right|
&\le \left|\int_M R_{\hat g_\delta}\varphi d\mu_{\hat g_\delta}-\langle R_{\hat g},\varphi\rangle\right|+|a|\left|\int_M\varphi d\mu_{\hat g_\delta}-\int_M\varphi d\mu_{\hat g}\right|\notag\\
&\le b_\delta \|\varphi\|_{W^{1,\frac{n}{n-1}}(M)}, \forall \varphi\in C^\infty(M),\forall\delta\in(0,\delta_0],
\end{align}
where $b_\delta$ is a positive function of $\delta$ which only depends on $a$, $n$, $\hat g$ and satisfies $\lim_{\delta\to0^+}b_\delta=0$.

By the condition on $R_{\hat g}$ we have 
\begin{equation}\label{d}
\langle R_{\hat g}-a,\varphi\rangle \ge 0, \forall \varphi\in C^\infty(M),\varphi\ge0.
\end{equation}

Let $g_{\delta }(t)$, $g(t)$ be the Ricci flow starting from the metric $\hat g_\delta$ and $\hat g$ respectively, and let $R_{g_{\delta}(t)}$, $R_{g(t)} $ and $d\mu_{g_{\delta}(T)}$, $d\mu_{g(t)}$ be the scalar curvature and the volume form of $g_{\delta}(t)$, $g(t)$ respectively. Take an arbitrary $T\in (0,T_0]$ and an arbitrary nonnegative function $\tilde \varphi\in C^\infty(M)$. For any $\delta\in(0,\delta_0]$, we let $\varphi_t$ be the solution of the conjugate heat equation taken with repsect to the family of metrics $g_\delta (t)$ and with $\varphi_t|_{t=T}=\tilde \varphi$, then $\varphi_0\ge0$, and by Proposition \ref{prop_varphit} (3), (\ref{d3}) and (\ref{d}), we have
\begin{align}\label{aaa}
\int_M (R_{g_{\delta}(T)}-a)\tilde\varphi d\mu_{g_{\delta}(T)}\ge \int_M (R_{\hat g_\delta}-a)\varphi_0 d\mu_{\hat g_\delta}\ge -b_\delta \|\varphi_0\|_{W^{1,\frac{n}{n-1}}(M)}.
\end{align}
By H\"older inequality and Proposition \ref{prop_varphit} (1) and (2), since $g_\delta(t),(\delta,t)\in (0,\delta_0]\times [0,T_0]$ and $h$ are uniformly equivalent, by H\"older inequality we have
\[\|\varphi_0\|_{W^{1,\frac{n}{n-1}}(M)}\le C(n,h,p)\|\varphi_0\|_{L^\infty(M)}+\|\nabla_{\hat g_\delta}\varphi_0 \|_{L^2(M)} \le C(n,h,A,p,\tilde\varphi),\]
and as mentioned in Remark \ref{limitf}, it is known in \cite{Si02} that for any fixed $T\in (0,T_0]$, $g_{\delta}(T)$ smoothly converges to $g(T)$ as $\delta$ tends to $0^+$.

Thus we can let $\delta\to0^+$ in inequality (\ref{aaa}), and get 
\begin{align*}
\int_M (R_{g(T)}-a)\tilde\varphi d\mu_{g(T)} \ge 0,\forall T\in(0,T_0],\forall \tilde\varphi \in C^\infty(M).
\end{align*}
Since $g(t)$ is a smooth metric for $t\in(0, T_0]$, we get that $R_{g(t)} \ge a$ pointwisely on $M$, which completes the proof of the theorem.
\end{proof}

Now we are ready to prove Theorem \ref{thmY2}.
\begin{proof}[Proof of Theorem \ref{thmY2}]
Recall that $R_{g(t)} $ satisfies the standard evolution equation
\begin{align}\label{Rt}
\partial_t R_{g(t)} =\Delta R_{g(t)}  +2|\Ric_{g(t)}|^2.
\end{align}
Noting that $R_g\ge 0$ in distributional sense, by Theorem \ref{thm1.1} we see that $R_{g(t)}\ge 0$ for any $t\in (0,T_0)$. There are two cases needed to be handled. 
First, if $R_{g(t)} \equiv 0$ on $M\times(0,T_0]$, then (\ref{Rt}) shows that $\Ric_{g(t)}\equiv 0$ on $M\times(0,T_0]$. Then by Ricci flow equation we have that $g(t)=g(s)$, $\forall t,s\in(0,T_0]$. By (1) of Theorem \ref{thm66.2} we directly have that $ (M,\hat g)$ is isometric to $(M,g(t))$. 
On the other hand, if $R_{g(t)} \not\equiv 0$ on $M\times(0,T_0]$, then by the maximal principle we get that $R_{g(t)} >0$ on $M$ for any $t\in (0,T_0)$. Since $g(t)$ is smooth and $R_{g(t)}>0$ pointwisely on $M$, the theorem holds. 
\end{proof}

\section{Further questions}
In this section, let us discuss some problems related to singular metrics. To study the scalar curvature of low-regularity metrics, Schoen proposed such a conjecture (see \cite{LiMa}):
\begin{conjecture}[Schoen]
Let $g$ be a $C^0$ metric on $M$ which is smooth away from a submanifold $\Sigma\subset M$ with codim$(\Sigma\subset M)\ge 3$, $\sigma(M)\le0$ and $R_g\ge0$ on $M\setminus \Sigma$, then $g$ smoothly extends to $M$ and $\Ric_g\equiv 0$.
\end{conjecture}
Gromov also considered scalar curvature with low-regularity metrics, see \cite{Gm}. For a Ricci flow approach, see also \cite{Bm}, \cite{PBG}. Motivated by their results, we will have the following natural question:

\textbf{Question:} Let $(M^n,g,\Sigma)$ be a compact manifold with $g\in C^0(M)$ and $g$ is smooth away from $\Sigma$, and $R_g\ge 0$ on $M\setminus \Sigma$. What is the condition on $\Sigma$ such that there exists a smooth metric on $M$ with nonnegative scalar curvature? What is the condition on $\Sigma$ such that the Ricci flow starting from $g$ has nonnegative scalar curvature for time $t>0$? 

\begin{remark}
	Comparing with the conjecture of Schoen and the structure of nonnegative scalar curvature in \cite{ScYa2}, one may expect that $\Sigma$ is a codimensional three submanifold for the above question. 
\end{remark}

\begin{remark}
	From our result and the main result in \cite{JSZ}, if we assume $g\in C^0\cap W^{1,p}(M)$ $(n<p\le \infty)$, the condition for $\Sigma$ should be $\cH^{n-\frac{p}{p-1}}(\Sigma)<\infty$ when $n<p<\infty$ or $\cH^{n-1}(\Sigma)=0$ when $p=\infty$.  
\end{remark}

\appendix
\section{Lower scalar curvature bounds on singular metrics}\label{sgl}
In this appendix, we check that if $g$ is $C^2$ away from a closed subset $\Sigma$ with $\cH^{n-\frac{p}{p-1}}(\Sigma)<\infty$ when $n<p<\infty$ or $\cH^{n-1}(\Sigma)=0$ when $p=\infty$, and $R_g\ge a$ pointwisely on $M\setminus\Sigma$, then $R_g\ge  a$ in distributional sense. Though it is proved in \cite{JSZ}, for the sake of completeness, we give a proof here.
\begin{lemma}\label{lmA1}
	Let $ M^n $ be a smooth manifold with $g\in C^0\cap W^{1,p}_{loc}(M)$ with $n\le p\le \infty$. Assume $g$ is smooth away from a closed subset $\Sigma$ with $\cH^{n-\frac{p}{p-1}}(\Sigma)<\infty$ if $n< p<\infty$ or $\cH^{n-1}(\Sigma)=0$ if $p=\infty$, and assume $R_g\ge a$ on $M\setminus \Sigma$ for some constant $a$, then $\langle R_g-a, u\rangle\ge 0$ for any nonnegative $u\in C^\infty(M)$. 
\end{lemma}
To prove the Lemma \ref{lmA1}, we need a standard cut-off function lemma (see \cite[Lemma A.1]{JSZ} and \cite{Cheeger}):
\begin{lemma}\label{l:cut-off}
	Let $(M^n,h)$ be a smooth manifold. Assume $\Sigma\subset M$ is a closed subset. Then there exists a sequence of cut-off function $\varphi_\epsilon$ of $\Sigma$ such that the following holds:
	\begin{itemize}
		\item[(1)] $0\le \varphi_\epsilon\le 1$ and $\varphi\equiv 0$ in a neighborhood of $\Sigma$ and $\varphi_\epsilon \equiv 1 $ on $M\setminus B_{\epsilon}(\Sigma)$.
		\item[(2)] If $\cH^{n-1}(\Sigma)=0$, then $\lim_{\epsilon\to 0}\int_M|\nabla \varphi_\epsilon|(x)dx=0.$
		\item[(3)] If $\cH^{n-p}(\Sigma)<\infty$ with $p>1$, then  $\lim_{\epsilon\to 0}\int_M|\nabla \varphi_\epsilon|^{p}(x)dx=0.$
	\end{itemize}
\end{lemma}
\begin{proof}[proof of Lemma \ref{lmA1}]

	Let $\eta_\epsilon\ge 0$ be a sequence of smooth cut-off functions of $\Sigma$ as in Lemma \ref{l:cut-off} such that 
	\begin{itemize}
		\item[(1)] $\eta_\epsilon\equiv 1$ in a neighborhood of $\Sigma$,
		\item[(2)] ${\rm supp ~}\eta_\epsilon \subset B_{\epsilon}(\Sigma)$ and $0\le \eta_\epsilon\le 1$.
		\item[(3)] $\lim_{\epsilon\to 0}\int_M|\nabla \eta_\epsilon|^{p/(p-1)}d\mu_h=0$. 
	\end{itemize}
	For any $u\in C^\infty(M)$ with $u\ge 0$, we have 
	\begin{align}
		\langle R_g-a,u\rangle=\langle R_g-a,\eta_\epsilon u\rangle+\langle R_g-a,(1-\eta_\epsilon)u\rangle.
	\end{align}
	Noting that for any $\epsilon>0$, we have 
	\begin{align}
		\langle R_g-a,(1-\eta_\epsilon)u\rangle=\int_{M\setminus \Sigma} (R_g-a)(1-\eta_\epsilon)ud\mu_g\ge   0. 
	\end{align}
	Thus to prove the lemma, it suffices to show that 
	\begin{align}\label{e:limitepsilonRgu0}
		\lim_{\epsilon\to 0}|\langle R_g-a,\eta_\epsilon u\rangle|=0.
	\end{align}
	Actually, noting that $u\in C^\infty(M)$, by definition, $	
		\langle R_g,u\rangle=\int_M \left(-V\cdot \tilde\nabla \left(u \frac{d\mu_g}{d\mu_h}\right)+Fu \frac{d\mu_g}{d\mu_h}\right)d\mu_h$, we can estimate 
		\begin{align*}
		|\langle R_g-a,\eta_\epsilon u\rangle|\le& \int_M |V|\cdot \left|\tilde\nabla (\eta_\epsilon u \frac{d\mu_g}{d\mu_h})\right|d\mu_h+\int_M |F-a|\cdot \eta_\epsilon u \frac{d\mu_g}{d\mu_h}d\mu_h\\
		\le &C\int_M |V||\bar \nabla \eta_\epsilon|d\mu_h +C \int_M |V|  |\tilde\nabla g| \eta_\epsilon d\mu_h+C\int_M |V| \eta_\epsilon d\mu_h+C \int_M|F-a		|\eta_\epsilon d\mu_h\\
		\le& C\left(\int_{M}|\tilde\nabla g|^pd\mu_h\right)^{1/p}\left(\int_M\left(|\tilde\nabla\eta_\epsilon|^{p/(p-1)}+|\eta_\epsilon|^{p/(p-1)}\right) d\mu_h\right)^{(p-1)/p}\\
		&+C\left(\left(\int_{M}|\tilde\nabla g|^p d\mu_h\right)^{2/p}+1\right)\left(\int_M|\eta_\epsilon|^{p/(p-2)}d\mu_h\right)^{(p-2)/p},
	\end{align*}
	here and below $C=C(n,p,h,g)$ will denote a positive constant independent of $\epsilon$ and varies from line to line.

	By H\"older inequality and Sobolev inequality, we have
	\begin{align*}
\left(\int_M|\eta_\epsilon|^{p/(p-1)}d\mu_h\right)^{(p-1)/p}\le C(h)\left(\int_M|\eta_\epsilon|^{p/(p-2)}d\mu_h\right)^{(p-2)/p}\le C(h) \left(\int_M|\tilde\nabla\eta_\epsilon|^{p/(p-1)} d\mu_h\right)^{(p-1)/p}.
	\end{align*}

	Thus we have
	\begin{align*}
|\langle R_g-a,\eta_\epsilon u\rangle|\le C\left(\int_M|\tilde\nabla\eta_\epsilon|^{p/(p-1)} d\mu_h\right)^{(p-1)/p}.
	\end{align*}

	Letting $\epsilon\to 0$, by the properties of $\eta_\epsilon$ we get \eqref{e:limitepsilonRgu0}. Thus we finish the proof of the lemma for $n< p<\infty$.\\
	When $p=\infty$, the arguement above still works. (See also  \cite{Lee2015a}.) Thus the lemma holds.
\end{proof}

\bibliographystyle{plain}

\begin{thebibliography}{CCT02}


\bibitem[ABN86]{ABN}A. D. Aleksandrov, V. N. Berestovskii, I. G. Nikolaev, Generalized Riemannian spaces, Uspekhi Mat. Nauk 41 (1986), no. 3(249), 3-44, 240.

\bibitem[Ba16]{Bm}R. Bamler,  A Ricci ﬂow proof of a result by Gromov on lower bounds for scalar curvature, Math. Res. Lett. 23 (2016), no. 2, 325-337.

\bibitem[BZ17]{BaZh}R. Bamler, Q. Zhang, Heat kernel and curvature bounds in Ricci flows with bounded scalar curvature. Adv. Math. 319 (2017), 396-450.

\bibitem[BZ19]{BaZh19}R. Bamler, Q. Zhang, Heat kernel and curvature bounds in Ricci flows with bounded scalar curvature-Part II. Calc. Var. Partial Differential Equations 58 (2019), no. 2, Paper No. 49, 14 pp.

\bibitem[BCW19]{BCW}R. Bamler, E. Cabezas-Rivas, B. Wilking, The Ricci flow under almost non-negative curvature conditions. (English summary)
Invent. Math. 217 (2019), no. 1, 95-126. 

\bibitem[BBI01]{BBI}D. Burago, Y. Burago, and S. Ivanov, A course in metric geometry, Graduate Studies in Mathematics, vol. 33, American Mathematical Society, Providence, RI, 2001.

\bibitem[BL08]{BL}K. Bruce, J. Lott, Notes on Perelman's papers. (English summary) 
Geom. Topol. 12 (2008), no. 5, 2587-2855.


\bibitem[Bu19]{PBG}P. Burkhardt-Guim, Pointwise lower scalar curvature bounds for $C^0$ metrics via regularizing Ricci flow. Geom. Funct. Anal. 29 (2019), no. 6, 1703-1772.

\bibitem[CZ06]{CZ}H. Cao, X. Zhu, A complete proof of the Poincaré and geometrization conjectures—application of the Hamilton-Perelman theory of the Ricci flow. Asian J. Math. 10 (2006), no. 2, 165-492.

\bibitem[CC97]{CC1}J. Cheeger, T. H. Colding, On the structure of spaces with Ricci curvature bounded below. I, J. Differential Geom. 46 (1997), no. 3, 406-480.

\bibitem[CC00a]{CC2}J. Cheeger, T. H. Colding
, On the structure of spaces with Ricci curvature bounded below. II, J. Differential Geom. 54 (2000), no. 1, 13-35.

\bibitem[CC00b]{CC3}J. Cheeger and T. H. Colding
, On the structure of spaces with Ricci curvature bounded below. III, J. Differential Geom. 54 (2000), no. 1, 37-74.

\bibitem[Ch03]{Cheeger} J. Cheeger, Integral bounds on curvature,
elliptic estimates, and rectifiability of singular sets, {\em Geom. Funct. Anal.} {\bf{13}} 20-72 (2003)

\bibitem[CJN21]{CJN21}J. Cheeger, W. Jiang, A. Naber, 
Rectifiability of singular sets of noncollapsed limit spaces with Ricci curvature bounded below. 
Ann. of Math. (2) 193 (2021), no. 2, 407-538.

\bibitem[CN12]{CN1}T. H. Colding and A. Naber, Sharp H\"older  continuity of tangent cones for spaces with a lower Ricci curvature bound and applications, Ann. of Math. (2) 176 (2012), no. 2, 1173-1229.

\bibitem[CN13]{CN2}J. Cheeger and A. Naber, Lower bounds on Ricci curvature and quantitative behavior ofsingular sets, Invent. Math. 191 (2013), no. 2, 321-339.


\bibitem[Gm07]{Gm07}M. Gromov, Metric structures for Riemannian and non-Riemannian spaces.  Modern Birkhäuser Classics. Birkhäuser Boston, Inc., Boston, MA, 2001.

\bibitem[Gm14]{Gm}M. Gromov, Dirac and Plateau billiards in domains with corners, Cent. Eur. J. Math. 12 (2014), no. 8, 1109-1156.

\bibitem[GL80]{GL80}M. Gromov, H. B. Lawson, Spin and scalar curvature in the presence of a fundamental group. I, Ann. of Math. (2) 111 (1980), no. 2, 209-230.

\bibitem[GT14]{grant2014positive}J. Grant,  N. Tassotti, A positive mass theorem for low-regularity Riemannian metrics. arXiv preprint arXiv:1408.6425, 2014.

\bibitem[Ha82]{Hamilton1982}R. Hamilton, Three-manifolds with positive Ricci curvature. Journal of Differential Geometry, 1982, 17(2): 255-306.

\bibitem[HH97]{han1997elliptic}Q. Han, F. Lin, Elliptic partial differential equations. American Mathematical Society, 1997.


\bibitem[Ji16]{Jiang}W. Jiang, Bergman kernel along the Kähler-Ricci flow and Tian's conjecture. J. Reine Angew. Math. 717 (2016), 195-226.

\bibitem[JN21]{JN21}W. Jiang, A. Naber, $L^2$ curvature bounds on manifolds with bounded Ricci curvature. Ann. of Math. (2) 193 (2021), no. 1, 107-222.


\bibitem[JWZ17]{JWZ}W. Jiang, F. Wang, X. Zhu, Bergman kernels for a sequence of almost Kähler-Ricci solitons. Ann. Inst. Fourier (Grenoble) 67 (2017), no. 3, 1279-1320.

\bibitem[JSZ21]{JSZ}W. Jiang, W. Sheng, H. Zhang, Removable singularity of positive mass theorem with continuous metrics. arXiv:2012.14041v1.

\bibitem[KW75]{KW75}J. L. Kazdan, F. W. Warner, Prescribing curvatures, Differential Geometry, Part 2, Amer. Math. Soc. Providence, R.I., 1975.


\bibitem[LL15]{Lee2015a}D. Lee, P. LeFloch, The positive mass theorem for manifolds with weak curvature. Communications in Mathematical Physics, 2015,  339(1): 99-120.

\bibitem[LM07]{LM07}P. LeFloch, C. Mardare, Definition and stability of Lorentzian manifolds with distributional curvature, Port. Math. (N.S.) 64 (2007), no. 4, 535-573.

\bibitem[LM19]{LiMa}C. Li, C. Mantoulidis, Positive scalar curvature with skeleton singularities. Math. Ann. 374 (2019), no. 1-2, 99-131.

\bibitem[LS15]{LS15}P. LeFloch, C. Sormani, The nonlinear stability of rotationally symmetric spaces with low regularity. (English summary)
J. Funct. Anal. 268 (2015), no. 7, 2005-2065.

\bibitem[LV09]{LV}J. Lott, C. Villani, Ricci curvature for metric-measure spaces via optimal transport, Ann. of Math. (2) 169 (2009), no. 3, 903-991.

\bibitem[MS12]{McSz}D. McFeron, G. Sz\'ekelyhidi, On the positive mass theorem for manifolds with corners. Comm. Math. Phys. 313 (2012), no. 2, 425-443.

\bibitem[Mi02]{Miao2002}P. Miao, Positive Mass Theorem on manifolds admitting corners along a hypersurface. Advances in Theoretical and Mathematical Physics, 2002, 6(6): 1163-1182.

\bibitem[MT07]{MT}J. Morgan, G. Tian, Ricci Flow and the Poincaré Conjecture. Clay Mathematics Institute. ISBN 978-0-8218-4328-4.



\bibitem[Pr02]{Pr1}G. Perelman, The entropy formula for the Ricci flow and its geometric applications, arXiv:math.DG/0211159.

\bibitem[Pr03]{Pr2}G. Perelman, Ricci flow with surgery on three-manifolds,  arXiv:math.DG/0303109.

\bibitem[Pr03]{Pr3}G. Perelman, Finite extinction time for the solutions to the Ricci flow on certain three-manifolds, arXiv:math.DG/0307245.

\bibitem[Sc89]{Sc}R. Schoen, Variational theory for the total scalar curvature functional for Riemannian metrics and related topics, Topics in calculus of variations. Springer, 1989.

\bibitem[SY79a]{Schoen1979}R. Schoen, S. T. Yau,  On the proof of the positive mass conjecture in general relativity. Communications in Mathematical Physics, 1979 65(1): 45-76.

\bibitem[SY79b]{ScYa2}R. Schoen, S. T. Yau,  On the structure of manifolds with positive scalar curvature. Manuscripta Math. 28 (1979), no. 1-3, 159-183.


\bibitem[Sh89]{Shi1989}W. Shi, Deforming the metric on complete Riemannian manifolds. Journal of Differential Geometry, 1989, 30(1): 223-301.

\bibitem[ST18]{ShTa}Y. Shi, L. Tam, Scalar curvature and singular metrics. Pacific J. Math. 293, 2018, no. 2, 427-470.

\bibitem[Si02]{Si02}M. Simon, Deformation of $C^0$ Riemannian metrics in the direction of their Ricci curvature. Communications in Analysis and
Geometry, 2002, 10(5): 1033-1074.

\bibitem[So21]{So}C. Sormani, et. al., Conjecture on convergence and scalar curvature, arXiv:2103.10093.

\bibitem[St06a]{St1}K.-T. Sturm, A curvature-dimension condition for metric measure spaces, C. R. Math. Acad. Sci. Paris 342 (2006), no. 3, 197-200.

\bibitem[St06b]{St2}K.-T. Sturm, On the geometry of metric measure spaces. I, Acta Math. 196 (2006),
no. 1, 65-131.

\bibitem[St06c]{St3}K.-T. Sturm, On the geometry of metric measure spaces. II, Acta Math. 196 (2006),
no. 1, 133-177.


















\end{thebibliography}

\end{document}